\documentclass[12pt,a4paper]{amsart}
\usepackage{a4wide}
\usepackage{amsfonts,amsthm,amsmath,amssymb,amscd,color}
\usepackage{graphicx}

\theoremstyle{plain}
\newtheorem{lem}{Lemma}
\newtheorem{prop}[lem]{Proposition}
\newtheorem{thm}[lem]{Theorem}
\newtheorem*{thm*}{Theorem}

\newtheorem*{cor*}{Corollary}

\newtheorem*{prop*}{Proposition}

\theoremstyle{definition}

\newtheorem*{defn*}{Definition}

\newtheorem*{ex*}{Example}
\newtheorem{rem}[lem]{Remark}
\newtheorem*{rem*}{Remark}

\theoremstyle{remark}

\DeclareMathOperator{\id}{id}
\DeclareMathOperator{\cl}{cl}

\DeclareMathOperator{\supp}{supp}

\newcommand{\1}{1\!\!1}
\renewenvironment{proof}{\medbreak{\noindent\em Proof }}{~{\hfill$\bullet$\bigbreak}}

\def\d{{\rm d}}

\begin{document}
 
\title[Stability of  iterated function systems on the circle]{Stability of  iterated function systems on the circle}
\author{Tomasz Szarek}
\author{Anna Zdunik}
\address{Tomasz Szarek, Institute of Mathematics, University of Gda\'nsk, Wita Stwosza 57, 80-952 Gda\'nsk, Poland}
\email{szarek@intertele.pl}
\address{Anna Zdunik, Institute of Mathematics, University of Warsaw,
ul.~Banacha~2, 02-097 Warszawa, Poland}
\email{A.Zdunik@mimuw.edu.pl}

\begin{abstract}
We prove that any Iterated Function System of circle homeomorphisms such that one of them has a dense orbit is asymptotically stable. The corresponding Perron-Frobenius operator is shown to satify the e-property, i.e. for any continuous function its iterates are  equicontinuous. The Strong Law of Large Numbers ({\bf SLLN}) for trajectories starting from an arbitrary point for such function systems is also proved.
\end{abstract}

\thanks{The research partially supported by the
  Polish NCN grants DEC-2012/07/B/ST1/03320 (T.S.) and 2014/13/B/ST1/04551 (A.Z.). T.S. was also supported by EC Grant RAQUEL}

\maketitle

\section{Introduction}

The action of discrete groups of homeomorphisms (diffeomorphisms) of the circle has been a subject of  intensive studies over last decades. See \cite{ghys} and \cite{navas2} for a detailed  description of recent results.

In this paper, we study the dynamics of discrete, finitely generated semigroups of orientation preserving circle homeomorphisms. After assigning a probability distribution over the set of generators, the system becomes an Iterated Function System. Iterated Function Systems were extensively studied because of their close connections to fractals  (see \cite{Barnsley}). We study spectral properties of the corresponding Markov operator $P$  and its  dual (transfer) operator $P^*$ acting on the space $C(\mathbf{S}^1)$  of continuous functions  on the circle. 
Note that such systems are  neither contracting, nor even contracting in average, so the well-known methods elaborated for contracting systems cannot be applied (see \cite{Barnsley, LY, sleczka, werner}).

On the other hand, if, instead of an action of  discrete (semi)groups, one considers random maps with  some absolutely continuous noise, several strong spectral properties of the corresponding transfer operator, including exponential decay of correlation, can be obtained (see, e.g., \cite{hom}).
However the case of action of discrete semigroups of homeomorphisms seems to be much more delicate.

We consider arbitrary finitely generated semigroups of orientation preserving homeomorphisms. The only restriction on the system which we assume for all our results is that one of the generators has dense orbits. Such systems will be called {\it admissible}. 

\

The  paper is organized as follows. In Section 2 we introduce the necessary notation, and we define the objects we deal with. We also introduce the notions of asymptotic stability and the e-property.

 Our first, preliminary  result is devoted to the uniqueness 
 of an invariant distribution. It is proved in Section 3.  We use here some ideas of Furstenberg (see \cite{Fur}) and Arnold and Crauel (see \cite{Ar} and the references therein). Recently they have been developed  by Kleptsyn {\it et al.}  in \cite{Navas} for proving uniqueness of an invariant measure for groups of circle homeomorphisms (see also \cite{DK}).
 
In Section 4  we  show some  auxiliary properties of semigroups of orientation preserving homeomorphisms which are useful in studying  asymptotic stability of our Iterated Function Systems.

Asymptotic stability is proved in Section 5. 

Section 6 is devoted to the proof of the e--property. Its usefulness in the study of asymp\-totic properties of Markov processes may be observed in \cite{KPS}. 

Finally, in Section 7 we show that any admissible Iterated Function System on the circle satisfies the Strong Law of Large Numbers ({\bf SLLN}) for trajectories starting from an arbitrary point.

\section{Notation}

Let $\mathbf{S}^1$ denote the circle with the counterclockwise orientation and let $\d (x, y)$ denote the normalized distance between $x, y\in \mathbf{S}^1$, so that the  length of the circle  equals $1$.
If $I$ is an arc in $\mathbf{S}^1$ then $|I|$ denotes the (normalized) length of $I$.
Let $x,y\in\mathbf{S}^1$ with $\d(x,y)<1/2$. By $[x,y]$ we denote the oriented (shorter) arc in $\mathbf{S}^1$, from  $x$ to $y$.
We define the following relation on $\mathbf{S}^1\times\mathbf{S}^1$: we say that $x< y$ 
if $\d(x,y)<1/2$ and the orientation of the arc $[x,y]$ coincides with the natural orientation on the circle.
Writing $x_1<x_2<\dots<x_M$ we assume that $\d(x_1, x_M)<1/2$ and $x_i<x_{i+1}, i=1,\dots M-1$.

 By $\mathcal B(\mathbf S^1)$ we denote the $\sigma$--algebra of Borel sets. Further, $C(\mathbf S^1)$ denotes the space of all continuous functions equipped with the supremum norm $\|\cdot\|$.
By $H^+$ we shall denote the set of all orientation preserving circle homeomorphisms.

By $\mathcal M_1$ and $\mathcal M_{fin}$ we denote the set of all Borel probability measures and Borel finite measures on $\mathbf S^1$, respectively. By $\supp\mu$ for $\mu\in\mathcal M_{fin}$ we denote the support of $\mu$.

An operator $P: \mathcal M_{fin}\to \mathcal M_{fin}$ is called a {\it Markov operator} if it satisfies the following two conditions:
\begin{itemize}
\item positive linearity: $P(\lambda_1\mu_1+\lambda_2\mu_2)=\lambda_1 P\mu_1+\lambda_2 P\mu_2$ for $\lambda_1, \lambda_2\ge 0$; $\mu_1, \mu_2\in\mathcal M_{fin}$;
\item preservation of the norm: $P\mu (\mathbf{S}^1)=\mu (\mathbf{S}^1)$ for $\mu\in\mathcal M_{fin}$.
\end{itemize}

A Markov operator $P$ is called a {\it Feller operator} if there is a linear operator $P^*: C(\mathbf S^1)\to C(\mathbf S^1)$ (dual to $P$) such that
$$
\int_{\mathbf S^1} P^* f(x)\mu (\d x)=\int_{\mathbf S^1} f(x) P\mu(\d x)\quad\text{for $f\in C(\mathbf S^1)$, $\mu\in\mathcal M_{fin}$.}
$$
Note that, if such an operator exists, then $P^*(\1)=\1$, hence $P^*(f)\ge 0$ if $f\ge 0$. As a consequence,
$$||P^*(f)||\le ||P^*(|f|)||\le ||f||,$$  
so $P^*$ is a continuous operator.
A measure $\mu_*$ is called {\it invariant} if $P\mu_*=\mu_*$. Since $\mathbf S^1$ is compact, every Feller operator on $\mathbf{S}^1$  has at least one invariant measure, by the Krylov--Bogolyubov theorem.

A Markov operator $P$ is called {\it asymptotically stable} if there exists a unique invariant measure $\mu_*\in\mathcal M_1$ such that
$$
\lim_{n\to\infty}\int_{\mathbf S^1} f(x)P^n\mu(\d x)=\int_{\mathbf S^1} f(x)\mu_*(\d x)
$$
for $f\in C(\mathbf S^1)$ and $\mu\in\mathcal M_1$.

%Let $P$ be a Feller operator. Observe that if for every $f\in C(\mathbf{S}^1)$  $\lim_{n\to\infty} P^{* n} f $ exists and is constant, then $P$ is asymptotically stable.

Following \cite{KPS}, we say that a Feller operator $P$ satisfies the {\it e--property} if for any $x\in\mathbf S^1$ and a continuous  function $f:\mathbf{S}^1\to\mathbb{R}$ we have 
%{\bf po co Lipschitz?? moze w przypadku niezwartym to jest istotne, ale w naszej sytuacji dowodzimy %przeciez dla kazdej funkcji ciaglej?}
$$
\lim_{y\to x} \sup_{n\in\mathbb N}|P^{* n}f (y) - P^{* n}f (x)|=0,
$$
i.e. if the family of iterates $\{P^{* n}(f): n\in\mathbb N\}$ is equicontinuous.
An equivalent notion  describing the e-property is the so--called almost periodicity of the dual operator $P^*$. Recall  that a  bounded linear operator $Q: F\to F$  of a
Banach space is called \emph{almost periodic} if for every $b\in F$  the sequence $(Q^n(b))_{n\in\mathbb N}$ is
relatively compact, that is, its closure in $F$  is compact in the norm topology. See, e.g., \cite{lyubich} or \cite{pubook}, for a description of spectral properties of almost periodic operators.

\vskip5mm

Let $\Gamma=\{g_1,\dots,g_k\}\subset H^+$ be a finite collection of homeomorphisms,  and let $(p_1,\dots p_k)$ be a probability distribution on $\{1,\dots,k\}$. Clearly, it defines a probability distrubution $p$ on   $\Gamma$, by putting $p(g_j)=p_j$. We assume that all $p_i$'s are strictly positive.
Put $\Sigma_n=\{1,\dots, k\}^n$,  and let $\Sigma_*=\bigcup_{n=1}^\infty\Sigma_n$ be the collection of all finite words with entries from $\{1,\dots ,k\}$.
For a sequence ${\bf i}\in\Sigma_*$, ${\bf i}=(i_1,\dots ,i_n)$, we denote by $|{\bf i}|$ its length (equal to $n$).  Finally, denote by $\Omega$ the infinite product  $\Omega=\{1,\dots,k\}^\mathbb{N}$. Let $\mathbb P$ be the product measure distribution on $\Omega$  generated by the initial distribution on $\{1,.\dots,k\}$.

 We consider the action of the semigroup generated by  $\Gamma$, i.e., the action of all compositions $g_{\bf i}=g_{i_n, i_{n-1},\ldots, i_1}=g_{i_n}\circ g_{i_{n-1}}\circ\cdots\circ g_{i_1}$, where ${\bf i}=(i_1,\dots i_n)\in\Sigma_*$.   The action of $\Gamma=\{g_1, g_2,\ldots, g_k\}$ is said to be {\it equicontinuous} if
the family of homeomorphisms $\{g_{\bf i}\}_{{\bf i}\in \Sigma_*}$ is equicontinuous. 
% for every $\delta>0$ there exists $\varepsilon>0$ such that if $\d (x, y)<\varepsilon$, then $\d %(g_{i_n, i_{n-1},\ldots, i_1}(x), g_{i_n, i_{n-1},\ldots i_1}(y))<\delta$ for any sequence $i_1, i_2, %\ldots, i_n\in\{1, 2, \ldots, k\}$ and $n\in\mathbb N$ (here $g_{i_n, i_{n-1},\ldots, i_1}=g_{i_n}\circ g_{i_{n-1}}\circ\cdots\circ g_{i_1}$). 
On the other hand, it is said to be {\it contractive} if for each $x\in\mathbf S^1$, there exists an open interval $I\subset\mathbf S^1$
containing $x$ and a sequence 
%$(i_1, i_2,\ldots)\in\{1, 2,\ldots, k\}^{\mathbb N}$ such that the length of the intervals $g_{i_n, %i_{n-1},\ldots, i_1}(I)$, denoted by $|g_{i_n, i_{n-1},\ldots, i_1}(I)|$, tends to $0$.
$({\bf i}_m)_{m\in\mathbb N}$ of elements of $\Sigma_*$ such that the length of the intervals $g_{{\bf i}_m}(I)$ tends to $0$ as $m\to\infty$.

The pair $(\Gamma, p)$ will be called an {\it Iterated Function System}.
The Markov operator $P: \mathcal M_{fin}\to\mathcal M_{fin}$ of the form
$$
P\mu=\sum_{g\in\Gamma} p(g)\mu\circ g^{-1},
$$
where $\mu\circ g^{-1}(A)=\mu (g^{-1}(A))$ for $A\in\mathcal B(\mathbf S^1)$, describes the evolution of distribution due to action of randomly chosen homeomorphisms from the collection $\Gamma$. 
It is a Feller operator, i.e., the operator $P^*: C(\mathbf S^1)\to C(\mathbf S^1)$ given by the formula
$$
P^*f(x)=\sum_{g\in\Gamma} p(g) f(g(x))\qquad\text{for $f\in C(\mathbf S^1)$ and $x\in\mathbf S^1$}
$$
is its dual.

We will say that an Iterated Function System $(\Gamma, p)$ is asymptotically stable if the corresponding Markov operator $P$ is asymptotically stable.

\section{Uniqueness of  an invariant measure}

The results in this section are in the spirit of \cite{Navas} (see also \cite{Fur}) but we have to point out one substantial difference.  Deroin {\it et al.} studied a group of circle homeomorphisms so that their family of transformations was richer than in our case. In particular for any homeomorphism the inverse of it, and, of course, the identity,  belonged to the class. Since we do not assume this, we have to slightly strengthen our assumption. Indeed, instead of assuming that the action of a semigroup is minimal, we suppose that the system is admissible.
The system $\Gamma=\{g_1, g_2,\ldots, g_k\}\subset H^+$ is said to be {\it admissible} if one of homeomorphisms, say $g_1$, is such that  $\{g_1^n(x): n\ge 1\}$ is dense  in $\mathbf{S}^1$ for some (and thus all) $x\in\mathbf S^1$. From now on we shall assume that if $\Gamma$ is admissible, then $g_1$ has a dense trajectory. If $\Gamma$ is admissible, then the pair $(\Gamma, p)$ will be called an admissible Iterated Function System.
\vskip3mm

We are in a position to formulate the main result of this section.

\begin{prop}\label{Prop1.9.1.15}Let $\Gamma=\{g_1, g_2,\ldots, g_k\}\subset H^+$ be admissible. Then the operator $P$ corresponding to $(\Gamma, p)$ admits a unique invariant measure for any distribution $p$.

\end{prop}

Before giving the proof of Proposition \ref{Prop1.9.1.15} we show the following lemma.

\begin{lem}\label{Lem1.26.12.15}
Let $\Gamma=\{g_1, g_2,\ldots, g_k\}\subset H^+$ be admissible. Then the action of $\Gamma$ is either equicontinuous or contractive.
\end{lem}

\begin{proof} 
 We start with a simple observation. Namely, if we assume that $\Gamma$ is admissible, then it is topologically conjugated to some $\tilde\Gamma\subset H^+$ such that an irrational rotation $\tilde g\in \tilde\Gamma$ is the conjugate of $g$. Thus, there is no loss of generality in assuming that $\Gamma$ contains an irrational rotation; say, $g_1$ is an irrational rotation.
 
Observe that if $\d (g_i(x), g_i(y))=\d(x, y)$ for $x, y\in\mathbf{S}^1$ and $i=1,\ldots, k$, then the action of $\Gamma$ is equicontinuous. Otherwise there is an arc $I=[a, b]\subset \mathbf{S}^1$ such that $|g(I)|<|I|$. Choose $\alpha\in (0, 1)$ such that $|g(I)|<\alpha |I|$. Let $\varepsilon>0$ be such that for $I'=[a, b+2\varepsilon]$ we have $|g(I')|\le\alpha|I|$. Fix $x\in\mathbf S^1$ and let  $J$ be an arbitrary arc with $|J|\le\varepsilon$. Since $g_1$ is an irrational rotation, there exist $m$ and $n_1,\ldots n_m\in\mathbb N$ such that
$g_1^{n_i}(J)\subset I'$ and $g_1^{n_i}(J)\cap g_1^{n_j}(J)=\emptyset$ for $i, j\in\{1,\ldots, m\}$ and, moreover, $\sum_{i=1}^m |g_1^{n_i}(J)|=m |J|\ge |I|$. Further $g\circ g_1^{n_1}(J)\cup\ldots\cup g\circ g_1^{n_m}(J)\subset g(I')$ and consequently $\sum_{j=1}^m |g\circ g_1^{n_j}(J)|\le \alpha |I|$.
Hence $|g\circ g_1^{n_j}(J)|\le\alpha |J|$ for some $j\in\{1,\ldots, m\}$. Otherwise, we would have
$m\alpha |J|<\sum_{j=1}^m |g\circ g_1^{n_j}(J)|\le \alpha |I|$, which is impossible. By induction we show that there is a sequence $(i_1, i_2,\ldots)\in\Sigma_\infty$ such that $|g_{i_n, i_{n-1},\cdots, i_1}(J)|\to 0$ as $n\to\infty$ and we are done. 
\end{proof}

\begin{proof} {\it of Proposition \ref{Prop1.9.1.15}}. As in the proof of Lemma~\ref{Lem1.26.12.15}, we can  assume that $\Gamma$ contains an irrational rotation.

Let $\Gamma=\{g_1,\ldots, g_k\}$ and let $g_1$ be an irrational rotation. 
If $\Gamma$ is equicontinuous, then the operator $P$ satisfies the e--property. Markov operators with the e--property have been already examined even in the setting of much more general phase spaces, i.e., general Polish spaces (see \cite{KapSlSz}). In particular, it was proved that such operators may have two different invariant measures $\mu_1$ and $\mu_2$ only if $\supp\mu_1\cap\supp\mu_2=\emptyset$ (see Theorem 1 in \cite{KapSlSz}). This may not be the case if $g_1$ is an irrational rotation for the support of any invariant measure is then equal to $\mathbf S^1$.

 Now let $\Gamma$ be contractive. Assume that uniqueness does not hold. Then there exists at least two different ergodic invariant measures. Again, since $\Gamma$ contains an irrational rotation $g_1$, every invariant measure is  supported on the whole circle $\mathbf{S}^1$. Let $\mu_1$ and $\mu_2$ be two ergodic invariant measures.  We shall prove that there exists a positive constant $\alpha$ and a measure $\nu$ such that $\mu_i\ge\alpha\nu$ for $i=1, 2$. Hence $\mu_1=\mu_2$ for the fact that two different ergodic invariant measures are mutually singular.
%Fix $x\in\mathbf {S}^1$ and 
Let $J\subset \mathbf {S}^1$ be such that $|g_{{\bf i}_m}(J)|\to 0$ as $m\to\infty$ for some sequence $({\bf i}_m)_{m\in\mathbb N}$ of elements  of $\Sigma_*$. 
%Since both measures  
Put 
%$\mu_1$ and $\mu_2$ must be supported on the whole circle  $\mathbf{S}^1$, we have 
$\alpha:=\min\{\mu_1(J), \mu_2(J)\}>0$. 
%Let $J_n=g_{i_n}\cdots g_{i_1}(J)$. 

Fix $\mu\in\{\mu_1, \mu_2\}$. Fix $f\in C(\mathbf S^1)$. We define a sequence of random variables   $(\xi_n^f)_{n\in\mathbb N}$ by the formula
$$
\xi_n^f(\omega)=\int_{\mathbf S^1} f(g_{i_1,\ldots, i_n}(x))\mu(\d x)\quad\text{for $\omega=(i_1, i_2,\ldots)$}.
$$
Since $\mu$ is an invariant measure for $P$, we easily check that $(\xi_n^f)_{n\in\mathbb N}$ is a bounded martingale. Note that this martingale depends on the  measure $\mu$.  From the Martingale Convergence Theorem it follows that $(\xi_n^f)_{n\in\mathbb N}$ is convergent $\mathbb P$-a.s. and since the space $C(\mathbf S^1)$ is separable, there exists a subset $\Omega_0$ of $\Omega$ with $\mathbb P(\Omega_0)=1$ such that $(\xi_n^f(\omega))_{n\in\mathbb N}$ is convergent for any $f\in C(\mathbf S^1)$ and $\omega\in\Omega_0$. Therefore for any $\omega\in\Omega_0$ there exists a measure $\omega(\mu)\in\mathcal M_1$ such that
$$
\lim_{n\to\infty}\xi_n^f(\omega)= \int_{\mathbf S^1} f(x)\omega(\mu)(\d x)\qquad\text{for every $f\in C(\mathbf S^1)$}.
$$

Now we are ready to show that for any $\varepsilon>0$ there exists $\Omega_{\varepsilon}\subset\Omega$ with $\mathbb P(\Omega_{\varepsilon})=1$ satisfying the following property:  for every $\omega\in\Omega_{\varepsilon}$ there exists an interval $I$ of length $|I|\le\varepsilon$ such that $\omega(\mu_i)(I)\ge\alpha$ for $i=1, 2$. 

Assume that this fact is proved. Then,  using additionally the compactness of $\mathbf{S}^1$,
we obtain ($\mathbb P$ a.s.) that there exists a point $\upsilon(\omega)\in\mathbf S^1$ such that  $\omega(\mu_i)\ge\alpha\delta_{\upsilon(\omega)}$ for $i=1, 2$, where $\delta_{\upsilon(\omega)}$ is the Diract delta measure supported at  $\upsilon(\omega)$. It is standard to show that the points $\upsilon (\omega)$, $\omega\in\Omega$, may be chosen in such a way that the function $\Omega\ni\omega\to\upsilon(\omega)\in\mathbf S^1$ is measurable. This will finish our proof. Indeed, define the measure $\nu\in\mathcal M_1$ by the formula 
$$
\nu:=\int_{\Omega}\delta_{\upsilon(\omega)} \mathbb P (\d\omega).
$$
and fix a non--negative function $f\in C(\mathbf S^1)$. We have
$$
\begin{aligned}
\int_{\mathbf S^1} f(x)\mu_i(\d x)&=\lim_{n\to\infty}\int_{\mathbf S^1} f(x)P^n\mu_i(\d x)
=\int_{\Omega}\lim_{n\to\infty} \xi_n^f(\omega) \mathbb{P}(\d \omega)\\
&\ge\alpha\int_{\Omega}f({\upsilon(\omega)})\mathbb{P}(\d \omega)=\alpha\int_{\mathbf S^1} f(x)\nu (\d x)
\end{aligned}
$$
for $i=1, 2$. Since $f\in C(\mathbf S^1)$ was an arbitrary non--negative continuous function, we obtain that $\mu_i\ge\alpha\nu$ for $i=1, 2$. 

We now complete our proof by constructing the claimed set $\Omega_\varepsilon$. This follows  some ideas of Deroin {\it et al.} (see \cite{Navas}).
Fix $\varepsilon>0$ and let $l\in\mathbb N$ be such that $2/l<\varepsilon$. 
Since
$|g_{{\bf i}_m}(J)|\to 0$, and $\Gamma$ contains an irrational rotation, we can require additionally  (modifying the sequence ${\bf i}_m$ if necessary)  that the arcs $J_m:=g_{{\bf i}_m}(J)$, where $m\le l$ are  
mutually disjoint. Put $n^*=\max_{m\le l}|{\bf i}_m|$.
% $|J_n|\to 0$ as $n\to\infty$ and $g_1$ is an irrational rotation, we may find sequences $(i_1^m,\ldots, i_{k_m}^m)\in\{1,\ldots, k\}^{k_m}$ such that
%$J_1,\ldots J_l$, where $J_m=g_{i_{k_m}^m}\cdots g_{i_1^m}(J)$ for $m=1,\ldots, l$, are mutually disjoint. 
Now observe that for any sequence ${\bf j}=(j_1,\ldots, j_n)\in \Sigma_*$ there exists $m\in\{1,\ldots, l\}$ such that $|g_{\bf j}(J_m)|<1/l<\varepsilon/2$.
%Now observe that for any sequence $(i_1,\ldots, i_n)$ there exists $m\in\{1,\ldots, l\}$ such that $|g_{i_1}\cdots g_{i_n}(J_m)|=|g_{i_1}\cdots g_{i_n}g_{i_{k_m}^m}\cdots g_{i_1^m}(J)|\le1/l<\varepsilon$.
This shows that for any cylinder  in $\Omega$, defined by fixing the first initial $n$ entries $(j_1,\ldots, j_{n})$, the conditional probability that  $(j_1,\ldots, j_n,\ldots, j_{n+k})$ are such that
$|g_{j_1, \ldots, {j_n},\ldots, j_{n+k}}(J)|\ge\varepsilon$ for all $k=1,\ldots, n^*$ is less than $1-q$ for some $q>0$. Hence there exists $\Omega_{\varepsilon}\subset\Omega$ with $\mathbb P(\Omega_{\varepsilon})=1$ such that for all $(j_1, j_2,\ldots)\in\Omega_{\varepsilon}$ we have $|g_{j_1,\ldots, j_n}(J)|<\varepsilon/2$ for infinitely many $n$'s. Since $\mathbf S^1$ is compact, we may additionally assume that for infinitely many $n$'s the set $g_{j_1,\ldots, j_n}(J)$ is contained in some set $I$ with $|I|\le\varepsilon$.
 This finishes the proof. \end{proof}

\section{Auxiliary results}

We start with the following lemma. Recall that we have normalized the arc length so that the length of the circle is equal to $1$.

\begin{lem}\label{comm}
Let $h$ be a circle orientation preserving homeomorphism. Assume that there exists $r<1$ such that $h$ maps every arc of length $r$ onto an arc of length at most $r$. If $r$ is irrational, then the map $h$ is a rotation and if  $r$
is rational, then $h$ commutes with the rotation by $r$ (denoted by $T_r$).
\end{lem}
\begin{proof}
First, assume that $r$ is irrational. Denote by $B$ the set of all $\beta\in (0, 1)$ such that  $h$ maps every arc of length $\beta$ onto an arc of length at most $\beta$. Then $r\in B$.
%We claim that the set $B$ is dense in  $[0,1]$. Indeed, 
It is easy to see that if $\beta\in B$ then $\{n\beta\}$  (the fractional part)  is also in $B$.   Since $r$ is irrational,  
the set $\{n\cdot r\}$ is dense in $[0,1]$.

By continuity of $h$ this implies that $h$ maps every arc $I$  onto an arc of length at most $|I|$.
Thus, $h$ must be an isometry -- a rotation.

Now, assume that  $r$ is rational, say $r=p /q$. Denote by $T_r$ the rotation by $r$.
Fix some $x_0\in S^1$ and the arc $I=[x_0, T_r(x_0)]$. 

Observe that $T_r^q=\id$ and the arcs 
$\bigcup_{i=1}^{q-1} T_r^i(I)$ cover the circle $\mathbf S^1$ exactly $p$ times. The same must be true for  
$\bigcup_{i=1}^{q-1} h(T^i_r(I))$. Since the length  of each arc  $h(T^i_r(I))$ is not bigger than the length of $T^i_r(I)$, it implies that $h$ maps every arc of length $r=p/ q$ onto an arc of the same length $r=p/ q$. Consequently, $h$ and $T_r$ commute:
$$h\circ T_r= T_r\circ h.$$
\end{proof}

\begin{prop}\label{stronglyexpansive} Let $\Gamma$ be contractive. Then there exists a rational rotation $R$ which commutes with the elements of $\Gamma$ and such that  for any $x\in\mathbf S^1$ and any closed arc $I\subset [x, R(x))$ there exists a sequence $({\bf i}_n)_{n\in\mathbb N}$ of elements of  $\Sigma_*$ such that $|g_{{\bf i}_n}(I)|\to 0$ as $n\to\infty$.
\end{prop}
\begin{proof}
We will call an arc $I$ contractible if there exists a sequence 
$({\bf i}_n)_{n\in\mathbb N}$ of elements of $\Sigma_*$ such that $|g_{{\bf i}_n}(I)|\to 0$ as $n\to\infty$.
%$(i_n)_{n\ge 1}\in\{1,\ldots, k\}^{\mathbb N}$ such that $|g_{i_n}\ldots g_{i_1}(I)|\to 0$ as $n\to\infty$. 
For $x\in \mathbf S^1$ we define $R(x)$ in the following way. Consider all positively oriented arcs $[x,y]$ such that the arc $[x,y]$ is also contractible. Clearly, if $[x,y]\subset [x,y']$ and $[x,y']$ is contractible then $[x,y]$ is contractible.
Keeping $x$ fixed, take the union of all contractible arcs $[x,y]$. This is an arc of length at least, say, $\varepsilon>0$.
%where $\varepsilon$ comes from Proposition ~\ref{Lem1.26.12.15}; 
One of its endpoints is $x$; the other endpoint is, by definition, $R(x)$. 
Note that it may happen that $R(x)=x$ (in this case the length of $[x, R(x)]$ is equal to the length of the circle).

For $x\in \mathbf S^1$ denote by $r(x)$ the length of the (positively oriented) arc $[x, R(x)]$.
Now, let $x,z,w\in\mathbf{S}^1$; we assume that   $x<w<z$.
Then $r(z)\ge r(w)-\d(z,w)$ and $r(w)\ge r(x)-\d(x,w)$. Consequently:
$$(r(z)+\d(z,x))-(r(w)+\d(w,x))=r(z)-r(w)+\d(z,x)-\d(w,x)=r(z)-r(w)+\d(z,w)\ge 0.$$
Putting  $\tilde r(z):=r(z)+\d(x,z)$, we thus have
$\tilde r(z)\ge \tilde r(w)\ge \tilde r(x)=r(x)$. Clearly, this implies that the function $r$ is Borel measurable.
Moreover, there exists a limit
$\lim_{z\to x^+} \tilde r(z)$, which satisfies
$\lim_{z\to x^+} \tilde r(z)\ge r(x)$, and, consequently, there exists a limit
$$\lim_{z\to x^+}r(z)\ge r(x).$$

Similarly, there exists a limit 
$\lim_{z\to x^-}  r(z)\le r(x)$. 
%Thus, the function $x\mapsto r(x)$ has countably many discontinuities.
%Chyba tez nie jest jasno, ale moze jakos sie wytlumacze...
%dalej tez trzeba by cos poprawic, bo napisalam ze r jest monotone, ale naprawde to $\tilde r$, i troche nie wiadomo co to znaczy..
%Chodzi o to ze r ma malo nieciaglosci.
%To moze napisac: r is constant almost everywhere, and since r satisfies, for every $x$
%$\lim_{z\to x^+}  r(z)\ge r(x)$, and  $\lim_{z\to x^-}  r(z)\le r(x)$,  both limits must be equal to %$r(x)$, so $r$ is continuous and, since it is $g_1$- invariant, it  is constant everywhere.
%Note that the function $x\mapsto r(x)$ is  non--decreasing in the following sense: if $x_n\to x$ so that %the arc $[x, x_n]$ is positively oriented then $r(x_n)\ge r(x)$. From monotonicity it follows that the unction $x\to r(x)$ has countably many discontinuities. 
%In particular, it is Borel measurable. 
Obviously, if  an arc $I$ is contractible then for every $g_i$, $i\in\{1,\ldots, k\}$, the arc $g_i^{-1}(I)$ is also contractible. Thus
\begin{equation}\label{monot}
r(g_1^{-1}(x))\ge r(x),
\end{equation}
where, recall, $g_1$ is an irrational rotation. Since $g_1$ is ergodic with respect to the Lebesgue measure, formula \eqref{monot} implies that the function $r$ is constant  (Lebesgue) almost everywhere in $\mathbf S^1$.
Since, for every $x$, we have 
$\lim_{z\to x^+}  r(z)\ge r(x)$, and  $\lim_{z\to x^-}  r(z)\le r(x)$,  both limits must be equal to $r(x)$, so the function $r$ is continuous and, since it is $g_1$- invariant, it  is constant everywhere.
%and since $r$ is monotone it is constant in $\mathbf S^1$.
So $r(x)\equiv r$  for some $r>0$. Consequently, for every $x\in \mathbf S^1$ the length of the arc $[x, R(x)]$ is equal to $r$.

Now,  choose $g\in\Gamma$ such that $g$ is not a rotation. If an arc $I$ is contractible then $g^{-1}(I)$ is also contractible. Hence,  $g^{-1}$ maps any arc of length $r$ onto an arc of length at most  $r$.  Therefore, $r$ is rational, by Lemma~\ref{comm}. 
From Lemma~\ref{comm} we conclude also that either $r=1$ and then $R=\id$ or $r<1$ and all elements of $\Gamma$ commute with $R:=T_r$.
The proof is complete.
\end{proof}

\begin{lem}
Let $\Gamma$ be contractive and let $R$ be a rational rotation that commutes with $\Gamma$. Then the unique invariant measure $\mu_*$ is $R$--invariant, i.e. $\mu_*\circ R^{-1}=\mu_*$.
\end{lem}

\begin{proof} We have
$$
\aligned
&P(\mu_*\circ R^{-1})(A)=\sum_{i=1}^k p_i\mu_*(R^{-1}g_i^{-1}(A))\\
&=\sum_{i=1}^k p_i\mu_*(g_iR)^{-1}(A))=\sum_{i=1}^k p_i\mu_*(g_i^{-1}(R^{-1}(A))=\mu_*(R^{-1}(A))=(\mu_*\circ R^{-1})(A)
\endaligned
$$
for any Borel set $A$. Since $P$ possesses a unique invariant measure we conclude that $\mu_*\circ R^{-1}=\mu_*$.
\end{proof}

The above lemma and the proof of Proposition \ref{Prop1.9.1.15}  easily imply the following.

\begin{prop}\label{order} The measure $\omega(\mu_*)$ for $\omega=(i_1, i_2,\ldots)$ defined $\mathbb P$-a.s. is $R$--invariant and, consequently, 
$$\omega(\mu_*)=\frac{1}{M}\sum_{m=0}^{M-1}\delta_{R^m(\upsilon(\omega))},$$
where $\upsilon(\omega)$ are the points defined in the proof of Proposition \ref{Prop1.9.1.15}  and $M$ is the order of the rotation $R$, i.e., the smallest integer such that $R^M={\rm Id}$.
\end{prop}

\section{Stability}
We start with an easy  criterion for stability when $\Gamma=\{g_1,\ldots, g_k\}$ is equicontinuous.

\begin{thm}
Let $\Gamma=\{g_1, \ldots, g_k\}$ be equicontinuous and let a probability distribution $p$ be given. Let $P$ be the Markov operator corresponding to $(\Gamma, p)$ and $P^*$ its dual. If for any $f\in C({\bf S}^1)$, $f\ge 0$ and $f\not\equiv 0$, there exists $r\in\mathbb N$
such that $P^{* r}f(x)>0$ for $x\in {\bf S}^1$, then the Iterated Function System $(\Gamma, p)$ is asymptotically stable.
\end{thm}

The proof follows from the results proved for almost periodic primitive operators (for details see Theorem 5.5.3 in \cite{pubook}. See also Theorem 6 in \cite{Walters}).

\begin{rem}\label{remark_on_equicont} Assume now that $\Gamma$ contains two rotations  $g_1=T_\alpha$ and   $g_2=T_\beta$ such that $(\alpha-\beta)$ is irrational. Then there exists $r\in\mathbb N$
such that $P^{* r}f(x)>0$ for $x\in {\bf S}^1$ and consequently the corresponding Iterated Function System $(\Gamma, p)$ is asymptotically stable for any probability distribution $p$.
\end{rem}

\begin{proof} Define the set  
$$
U_{m, x}:=\{T_{k\alpha+(m-k)\beta}(x): k=0,\dots m \}=\{T_{k(\alpha-\beta)}(T_{m\beta}(x)): k=0,\dots m\}
$$
for $x\in{\mathbf S}^1$ and $m\in\mathbb N$. 
Fix $f\in C({\mathbf S}^1)$ and observe that to prove that $P^{* r}f(x)>0$ for $x\in {\mathbf S}^1$ and some $r\in\mathbb N$ it is enough to show that for every $\varepsilon>0$ and $x\in {\mathbf S}^1$ there is $m\in\mathbb N$ such that $U_{m, x}$ forms an $\varepsilon$--net  in $\mathbf{S}^1$. Indeed, since $f$ is continuous and $f\not\equiv 0$, there exists $\varepsilon >0$ such that for any $\varepsilon$--net there is $y$ from this net such that $f(y)>0$. Further, if $U_{m, x}$ is some $\varepsilon$--net, then
$f(T_{i\alpha+(m-i)\beta}(x))>0$ for some $i\in\{0,\ldots, m\}$ and consequently $P^{*m} f(x)\ge p_1^i p_2^{m-i} f(T_{i\alpha+(m-i)\beta}(x))>0$.

Denoting now by  $\gamma:=\alpha-\beta$, the problem reduces to the following immediate observation.
Assume that  $\gamma$ is irrational.
Then for every $\varepsilon>0$ there exists $m\in\mathbb N$ such that for every $y\in\mathbf{S}^1$ the set

$$
\{T_{k\gamma}(y), k=0,\dots m\}
$$
forms an $\varepsilon$--net  in $\mathbf{S}^1$. \end{proof}

Now, assume that  $\Gamma$ is equicontinuous, and let $g_1\in \Gamma$ be a hoemomorphism with dense orbits. After the appropriate change of variables we can assume that $g_1$ is an irrational rotation. Using Lemma ~\ref{Lem1.26.12.15} we see that all elements of $\Gamma$ are rotations.
So we can reformulate Remark ~\ref{remark_on_equicont} as follows.
\begin{rem} Assume that  $\Gamma$ contains an element with dense orbits, and that $\Gamma$ is equicontinuous. If, for some $g_i, g_j\in \Gamma$ the homeomorphism $g_i^{-1}\circ g_j$ has dense orbits then for any probability distribution $p$ the Iterated Function System $(\Gamma, p)$ is asymptotically stable.
\end{rem}

The following lower bound criterion for stability of Markov operators generalizing Doeblin's theorem (see \cite{Doeblin}) will be useful in proving stability when the family $\Gamma$ is contractive. 

\begin{thm}\label{THM1.6.3.15} Let $P$ be an arbitrary Markov operator. Assume that for any $\varepsilon>0$ and $f\in C(\mathbf S^1)$ there exists $\alpha>0$ such that for every $\mu_1, \mu_2\in\mathcal M_1(\mathbf S^1)$ there are $\nu_1, \nu_2\in\mathcal M_1(\mathbf S^1)$ and $N\in\mathbb N$ satisfying
\begin{equation}\label{4.3.15.01}
P^N\mu_i\ge\alpha\nu_i\qquad\text{for $i=1, 2$}
\end{equation}
and
$$
\limsup_{n\to\infty} \left|\int_{\mathbf S^1} f(x)P^n\nu_1 (\d x) -\int_{\mathbf S^1} f(x)P^n\nu_2 (\d x)\right|\le\varepsilon.
$$
Then the operator $P$ is asymptotically stable.
\end{thm}

\begin{proof} Fix $\varepsilon>0$ and $f\in C(\mathbf S^1)$.  Fix $\mu_1, \mu_2\in\mathcal M_1(\mathbf S^1)$. We shall show that
\begin{equation}\label{4.3.15.03}
\left|\int_{\mathbf S^1} f\d P^n\mu_1-\int_{\mathbf S^1} f\d P^n\mu_2\right|\le 2\varepsilon
\end{equation}
for all $n$ sufficiently large. To do this choose $k\in\mathbb N$ such that $2(1- \alpha)^k\|f\|<\varepsilon$, where $\alpha>0$ is such that condition (\ref{4.3.15.01}) holds for given $\varepsilon$ and $f$.
Using \eqref{4.3.15.01} we have $P^{N_1}\mu_i=\alpha\nu_i+(1-\alpha)\tilde\mu_i^1$, where $\tilde\mu_i^1$ are some probability measures. Proceeding inductively, and using  \eqref{4.3.15.01}  at every step
we may find sequences of probability measures $\nu_i^1,\ldots, \nu_i^k$ and $\tilde\mu_i^k$   for $i=1, 2$ and a sequence of positive integers $N_1, \ldots, N_k$ such that
$$
\begin{aligned}
P^{N_1+\cdots +N_k}\mu_i&=\alpha P^{N_2+\cdots+N_k}\nu_i^1+\alpha (1-\alpha) P^{N_3+\cdots+N_k}\nu_i^2\\
&+\cdots+\alpha(1-\alpha)^{k-1}P^{N_k}\nu_i^k+(1-\alpha)^k\tilde\mu_i^k\qquad\text{for $i=1, 2$}
\end{aligned}
$$
and
$$
\limsup_{n\to\infty} \left|\int_{\mathbf S^1} f(x)P^n\nu_1^j (\d x) -\int_{\mathbf S^1} f(x)P^n\nu_2^j (\d x)\right|\le\varepsilon.
$$
for $j=1,\ldots, k$ (for details see  Theorem 5.3 in \cite{SzarDiss}). Hence, by the definition of $k$, condition (\ref{4.3.15.03}) follows. Since $P$ admits an invariant measure, the proof is complete.
\end{proof}

\begin{lem}\label{l1.6.3.15} Let $\Gamma=\{g_1, \ldots, g_k\}$ be admissible and contractive. For any $n\in\mathbb N$ and $x, y\in\mathbf S^1$ there exists $N\in\mathbb{N}$ and  two collections of elements of $\Sigma_N$: ${\bf i}_1,\dots, {\bf i}_n$, ${\bf j}_1,\dots, {\bf j}_n$  such that
%for $i=1,\ldots, n$ and $j=1, 2$ such that
$$
g_{{\bf i}_1}(x)< g_{\bf{j}_1}(y) < g_{{\bf i}_2}(x) < g_{{\bf j}_2}(y) < \cdots < g_{{\bf i}_n}(x)<g_{{\bf j}_n}(y).
$$
Moreover, one can require that  the length of the arcs $(g_{{\bf i}_m}(x), g_{{\bf j}_m}(y))$ and $(g_{{\bf j}_m}(y), g_{{\bf i}_{m+1}}(x))$ for $m=1,\ldots, n$ is bigger than some constant, say $\tau$, depending on $n$ but independent of $x$ and $y$.
\end{lem}

\begin{proof}  After a necessary change of variables one can assume that  $g_1\in\Gamma$ is an irrational rotation. Fix $x, y\in \mathbf S^1$ and let $I$ be such an arc that $x\in I$ and $|g_{{\bf l}_m}(I)|\to 0$ as $m\to\infty$ for some sequence $({\bf 
l}_m)_{m\in\mathbb N}$ of elements of $\Sigma_*$. Since $g_1$ is an irrational rotation,  we can additionally assume that  there is an open arc $I_0$, with $I\setminus {\rm 
cl}I_0\neq\emptyset$ such that  $g_{{\bf l}_m}(I)\subset I_0$  for all $m\in\mathbb N$. We shall proceed by induction on $n$.
The case $n=1$ is obvious but we show that we may additionally require that $g_{{\bf i}_1}(x), g_{{\bf j}_1}(y)\in I$. Indeed, 
set $h_1:=g_{{\bf l}_m}$, $h_2=g_1^{|{\bf l}_m|}$ and observe that $h_2$ is an irrational rotation again. Hence there exists $l\in\mathbb N$ such that $h_2^l(y)\in I\setminus\cl I_0$ and $h_2^l(y)>z$ for every $z\in I_0$. Since $h_1^l(x)\in I_0$, we are done.

Now let the statement of our lemma hold for some $n$. Denote by $\tilde I=(g_{{\bf i}_1}(x), g_{{\bf j}_n}(y))\subset I$. We have $h_1(\tilde I)\subset I_0$. Analogously as in the previous step we find  $l$ such that $h_2^l(x)\in I\setminus\cl I_0$ and $h_2^l(x)>z$ for every $z\in I_0$. Hence there exist $N$ and $\tilde{\bf i}_m\in \Sigma_N$, $m=1,\ldots, n+1$,  $\tilde{\bf j}_m\in \Sigma_N$, $m=1,\ldots, n$ such that
$$
g_{\tilde{\bf i}_1}(x)< g_{\tilde{\bf j}_1}(y) < g_{\tilde{\bf i}_2}(x) < g_{\tilde{\bf j}_2}(y) < \cdots <  g_{\tilde{\bf j}_n}(y)<g_{\tilde{\bf i}_{n+1}}(x).
$$

Replacing $\tilde I$ with $\hat I= (g_{{\bf i}_1}(x), g_{{\bf i}_{n+1}}(x))$ and $x$ with $y$, and repeating the above procedure we show our lemma for $n+1$. 

Now we observe that the length of the arcs $(g_{{\bf i}_m}(x), g_{{\bf j}_m}(y))$ and $(g_{{\bf j}_m}(y), g_{{\bf i}_{m+1}}(x))$ for $m=1,\ldots, n$ may be bigger than some constant $\tau$ (depending on $n$) but independent of $x, y\in\mathbf S^{1}$. Fix $z\in\mathbf S^1$. From what we have proved above it follows that for any $u, v\in\mathbf S^1$ there exist two sequences $({\bf p}_n)_{n\in\mathbb N}$
and $({\bf q}_n)_{n\in\mathbb N}$ of elements of $\Sigma_*$ such that $|{\bf p}_n|=|{\bf q}_n|$ for $n\in\mathbb N$ and 
$\d (g_{{\bf p}_n}(u), g_{{\bf q}_n}(v))\to 0$ as $n\to\infty$. Since $g_1$ is an irrational rotation we may assume additionally that 
$g_{{\bf p}_n}(u)\to z$ and $g_{{\bf q}_n}(v)\to z$ as $n\to\infty$. Applying now first part of our consideration  we obtain that
the constant $\tau$, chosen for $x=y=z$, will be also a lower bound for the length of the arcs $(g_{{\bf\tilde i}_m}(u), g_{{\bf \tilde j}_m}(v))$ and $(g_{{\bf \tilde j}_m}(v), g_{{\bf \tilde i}_{m+1}}(u))$ for $m=1,\ldots, n$ and some ${\bf\tilde i}_1,\dots, {\bf\tilde i}_n$, ${\bf\tilde j}_1,\dots, {\bf\tilde j}_n$ (with the same length) such that
$$
g_{{\bf\tilde i}_1}(u)< g_{{\bf\tilde j}_1}(v) < g_{{\bf\tilde i}_2}(u) < g_{{\bf\tilde j}_2}(v) < \cdots < g_{{\bf\tilde i}_n}(u)<g_{{\bf\tilde j}_n}(v).
$$
This completes the proof.
\end{proof}

\begin{lem}\label{l2.6.3.15} Let $\Gamma$ be admissible and contractive and let $(\mu^K_1)_{K\in\mathbb N}$ and $(\mu^K_2)_{K\in\mathbb N}$ be two sequences of probability distributions such that for any $K\in\mathbb N$  the measures $\mu^K_1$ and $\mu^K_2$ are uniformly distributed on $\{x_1, x_2,\ldots, x_K\}\subset\mathbf S^1$ and $\{y_1, y_2\ldots, y_K\}\subset\mathbf S^1$, respectively and
$$
x_1<y_1<x_2<y_2<\cdots < x_K<y_K.
$$
Let a probability distribution $p$ be given and let $P$ be the Markov operator corresponding to the Iterated Function System $(\Gamma, p)$. 
Then for an  arbitrary $f\in C(\mathbf S^1)$ we have
\begin{equation}\label{e1.6.3.15}
\lim_{K\to\infty}\limsup_{n\to\infty}\left |\int_{\mathbf S^1} f(x)P^n\mu^K_1(\d x)-\int_{\mathbf S^1} f(x)P^n\mu^K_2(\d x)\right|=0.
\end{equation}
\end{lem}

\begin{proof} Fix $f\in C(\mathbf S^1)$ and  let $\omega=(i_1, i_2,\ldots)$ be such that $\omega(\mu_*)$ is defined. Let $z\in {\mathbf S^1}$ be such $\omega(\mu_*)=(\delta_z+\delta_{R(z)}+\ldots +\delta_{R^{M-1} (z)})/M$, by Proposition \ref{order}. Fix $\varepsilon>0$ and let $U_1,\ldots U_{M-1}$ be mutually disjoint open arcs such that $R^i(z)\in U_i$ and $|f(u)-f(v)|\le\varepsilon/2$ for $u, v\in U_i$,  $i=0, \ldots, M-1$. Fix $K\in\mathbb N$. All the arcs $(x_i, y_i)$ and $(y_i, x_{i+1})$ for $i=1,\ldots, K$ have positive $\mu_*$ measure, by the fact that the support of $\mu_*$ is equal to $\mathbf S^1$. Since $\mu_*\circ (g_{i_1}\cdots g_{i_n})^{-1}$ converges weakly to $\omega(\mu_*)$ and 
\begin{equation}\label{e1.16.11.15}
\omega(\mu_*)(\bigcup_{j=0}^{M-1} U_j)=1,
\end{equation}
for any arc $(u, v)$  and $n$ arbitrary large (depending on $\mu_*((u, v))$) there exists $w\in (u, v)$ such that 
$g_{i_1}\cdots g_{i_n}(w)\in \bigcup_{j=0}^{M-1} U_j$. If this is not the case, we obtain that $\omega(\mu_*)(\bigcup_{j=0}^{M-1} U_j)\le 1-\mu_*((u, v))$, contrary to condition (\ref{e1.16.11.15}).

Due to the above observation we find  points $z^n_i, w^n_i$ such that
$$
x_1<z^n_1<y_1<w^n_1<x_2<z^n_2<y_2<w^n_2<\cdots <x_K<w^n_K<y_K
$$
and 
$$
g_{i_1}\cdots g_{i_n}(z^n_i), g_{i_1}\cdots g_{i_n}(w^n_i)\in\bigcup_{j=0}^{M-1} U_j
$$
for $i=1,\ldots, K$ and all $n\ge n_0$, where $n_0$  depends on the measure  $\mu_*$ (or: on length) of the arcs $(x_1, y_1), (y_1, x_2),\ldots, (y_{K-1}, x_K), (x_K, y_K)$. In other words, it is independent of locations of the points if the distance between them is bigger than some $\tau>0$. 
From this it follows that there are at most $2M$ pairs of $(x_i, y_i)$ such that $g_{i_1}\cdots g_{i_n}(x_i), g_{i_1}\cdots g_{i_n}(y_i)$ are not in the same $U_j$.
Hence
$$
\left|\int_{\mathbf S^1} f(x) \mu_1^K\circ (g_{i_1}\cdots g_{i_n})^{-1}(\d x) - \int_{\mathbf S^1} f(x) \mu_2^K\circ (g_{i_1}\cdots g_{i_n})^{-1}(\d x)\right|\le \varepsilon/2+2\|f\| M/K.
$$
for all $n\ge n_0$. We may find $n_0$ such that the above condition holds for all $\omega$ from some set $\tilde\Omega$ with $\mathbb P(\tilde\Omega)\ge 1-\varepsilon/(2M)$. Then we obtain
$$
\left |\int_{\mathbf S^1} f(x)P^n\mu^K_1(\d x)-\int_{\mathbf S^1} f(x)P^n\mu^K_2(\d x)\right|\le\varepsilon+2\|f\| M/K
$$
for all $n\ge n_0$. Taking limit as $K\to\infty$ completes the proof.
\end{proof}
\begin{rem}\label{remark}
The above proof also shows the following: Fix some $f\in C(\mathbf{S}^1)$. For every $\varepsilon>0$ and $\tau>0$ there exists   $n_0=n_0(\tau,\varepsilon, f)$ such that
$$
\left |\int_{\mathbf S^1} f(x)P^n\mu^K_1(\d x)-\int_{\mathbf S^1} f(x)P^n\mu^K_2(\d x)\right|\le \varepsilon +2\|f\| M/K
$$
for $n\ge n_0$ and any probability measures $\mu_1^K, \mu_2^K$ distributed  uniformly on  $\{x_1, x_2,\ldots, x_K\}$ and $\{y_1, y_2\ldots, y_K\}$, respectively, such that
$$
x_1<y_1<x_2<y_2<\cdots < x_K<y_K,
$$
and the length (measure  $\mu_*$) of the arcs $(x_1, y_1), (y_1, x_2),\ldots, (y_{K-1}, x_K), (x_K, y_K)$ is bounded from below by $\tau>0$ (Here and below $M$ denotes the constant defined in Proposition \ref{order}).
\end{rem}

\begin{thm}\label{MTh} Let $\Gamma=\{g_1, \ldots, g_k\}$ be admissible and contractive and let $p$ be a probability distribution.
Then the Iterated Function System $(\Gamma, p)$ is asymptotically stable.
\end{thm}

\begin{proof} Recall that we may assume that $g_1$ is an irrational rotation. We are going to show that the assumptions of Theorem \ref{THM1.6.3.15} are satisfied.
Fix $\varepsilon>0$ and let $K\in\mathbb N$  be so large that
$$2\|f\| M/K<\varepsilon/2.$$
%$$
%\limsup_{n\to\infty}\left |\int_{\mathbf S^1} f(x)P^n\mu^K_1(\d x)-\int_{\mathbf S^1} %f(x)P^n\mu^K_2(\d x)\right|<\varepsilon/2.
%$$
Let  $(x, y)\in {\mathbf S^1}\times{\mathbf S^1}$. It follows from Lemma \ref{l1.6.3.15}  that 
there exists $N=N_{x, y}$ and ${\bf i}_l,{\bf j}_l\in\Sigma_N$ for $l=1,\ldots, K$ and  such that
$$
g_{{\bf i}_1}(x)< g_{{\bf j}_1}(y) < g_{{\bf i}_2}(x) < g_{{\bf j}_2}(y) < \cdots < g_{{\bf i}_K}(x)<g_{{\bf j}_K}(y).
$$
Further, we may find open neighbourhood $U_x$, $U_y$ of $x, y$, respectively, such that 
$$
g_{{\bf i}_1}(\tilde x)< g_{{\bf j}_1}(\tilde y) < g_{{\bf i}_2}(\tilde x) < g_{{\bf j}_2}(\tilde y) < \cdots < g_{{\bf i}_K}(\tilde x)<g_{{\bf j}_K}(\tilde y)\quad\text{for $(\tilde x, \tilde y)\in \overline{U}_x\times \overline{U}_y$}.
$$

%$$
%g_{\omega^1_1}(\tilde x)< g_{\omega^1_2}(\tilde y) < g_{\omega^2_1}(\tilde x) < g_{\omega^2_2}(\tilde y) < \cdots < g_{\omega^K_1}(\tilde %x)<g_{\omega^K_2}(\tilde y)\quad\text{for $(\tilde x, \tilde y)\in \overline{U}_x\times \overline{U}_y$}.
%$$
Now, choose a collection of pairs  $(x_1, y_1),\ldots, (x_r, y_r)\in {\mathbf S^1}\times{\mathbf S^1}$ such that ${\mathbf S^1}\times{\mathbf S^1}=\bigcup_{i=1}^r U_{x_i}\times U_{y_i}$. 
Fix $\mu_1, \mu_2\in\mathcal M_1(\mathbf S^1)$. Set $p=:\min_i p_i$. Take the product measure  $\mu_1\times\mu_2$ and observe that there exists $j\in \{1,\dots,  r\}$ such that
$(\mu_1\times\mu_2)(U_{x_{j}}\times U_{y_{j}})\ge 1/r$. It is now easy to check that
\begin{equation}\label{first}
P^{N_{x_{j}, y_{j}}}\mu_1\ge p^{N_{x_{j}, y_{j}}}K\int_{U_{x_{j}}}m_1(x)\mu_1(\d x)
\end{equation}
and
\begin{equation}\label{second}
P^{N_{x_{j}, y_{j}}}\mu_2\ge p^{N_{x_{j}, y_{j}}}K\int_{U_{y_{j}}}m_2(y)\mu_2(\d y),
\end{equation}
where $m_1(x)$ and $m_2(y)$ are probability measures as in the hypothesis of Lemma \ref{l2.6.3.15}, i.e., the measures uniformly distributed over the 
points $$g_{{\bf i}_1}(x), g_{{\bf i}_2} (x),\dots, g_{{\bf i}_K}(x)$$
and  $$g_{{\bf j}_1}(y), g_{{\bf j}_2} (y),\dots, g_{{\bf j}_K}(y),$$ 
respectively.  Set
$$
\nu_1(\cdot)=\frac{\int_{U_{x_{j}}}m_1(x)(\cdot)\mu_1(\d x)}{\mu_1(U_{x_j})}\quad\text{and}\quad\nu_2(\cdot)=\frac{\int_{U_{y_{j}}}m_2(x)(\cdot)\mu_2(\d x)}{\mu_2(U_{y_j})}
$$
and observe that estimates \eqref{first} and \eqref{second} can be rewritten as:
$$
P^{N_{x_{j}, y_{j}}}\mu_i\ge \alpha\nu_i\qquad\text{for $i=1, 2$}
$$
with $\alpha:=p^N K r^{-1}$, where $N=\max_{1\le l\le r} N_{x_l, y_l}$. Since, by Lemma~\ref{l2.6.3.15} and Remark~\ref{remark}  
$$
\left|\int_{\mathbf S^1} f(z)P^n m_1(x)(\d z) -\int_{\mathbf S^1} f(z)P^n m_2 (y)(\d z)\right|\le\varepsilon
$$
for $(x, y)\in U_{x_{j}}\times U_{y_{j}}$ and $n$ sufficiently large, we obtain that
$$
\limsup_{n\to\infty} \left|\int_{\mathbf S^1} f(x)P^n\nu_1 (\d x) -\int_{\mathbf S^1} f(x)P^n\nu_2 (\d x)\right|\le\varepsilon.
$$
This completes the proof.
\end{proof}

\section{E--property}

The e--property  plays an important role in proving asymptotic properties of Markov processes. Usually it is a necessary step when we want to justify that the studied process has a unique invariant measure and is asymptotically stable. Here we inverted the order of our considerations. First we showed the uniqueness of an invariant measure and its stability and now we shall prove independently that  our operator also satisfies the e--property. This fact seems to be surprising because the transformations under considerations are neither contractions, nor average contractions. Here the e--property is forced by the geometry.

\begin{prop}
Let $\Gamma=\{g_1, \ldots, g_k\}$ be admissible and let $p$ be a probability distribution.
Then the operator $P$ corresponding to the Iterated Function System $(\Gamma, p)$ satisfies the e--property.
\end{prop}

\begin{proof} If $\Gamma=\{g_1, \ldots, g_k\}$ is equicontinuous, the e--property follows immediately. So we may assume that $\Gamma$ is contractive. Fix a function $f\in C(\mathbf{S}^1)$  and $\varepsilon>0$. Let $K\in\mathbb N$ be such that $2\|f\| M/K\le\varepsilon/2$. (Here again $M$ is the constant coming from Proposition ~\ref{order}). Take an arbitrary point  $z\in\mathbf S^1$.
% corresponding to $\{g_1, \ldots, g_k\}$ such that $\omega(\mu_*)$ is supported on $M$ Dirac delta). 
 From the proof of Theorem \ref{MTh}, applied for both $x:=z$ and $y:=z$  (see also Lemma~\ref{l1.6.3.15}), it follows that  we may find $\alpha>0$ and $N_1, \ldots, N_m$   such that $P^{N_1+\cdots +N_m}\delta_{z}$ admits two representations below:
\begin{equation}\label{e1.17.04.15}
\begin{aligned}
P^{N_1+\cdots +N_m}\delta_{z}&=\alpha P^{N_2+\cdots+N_m}\nu_1^1+\alpha (1-\alpha) P^{N_3+\cdots+N_m}\nu_1^2\\
&+\cdots+\alpha(1-\alpha)^{m-1}P^{N_m}\nu_1^m+(1-\alpha)^m\mu_1\\
P^{N_1+\cdots +N_m}\delta_{z}&=\alpha P^{N_2+\cdots+N_m}\nu_2^1+\alpha (1-\alpha) P^{N_3+\cdots+N_m}\nu_2^2\\
&+\cdots+\alpha(1-\alpha)^{m-1}P^{N_m}\nu_2^m+(1-\alpha)^m\mu_2,
\end{aligned}
\end{equation}
where, for every $j=1,\dots m$  the pair of probability measures   $(\nu_1^j, \nu_2^j)$  is  a convex combination of pairs of measures such that each  pair is uniformly distributed  over some collections of points $\{x_1,\ldots, x_K\}$ and $\{y_1,\ldots, y_K\}$, respectively, and
$$
x_1<y_1<x_2<y_2<\cdots < x_K<y_K,
$$
where the length of the arcs $(x_1, y_1), (y_1, x_2),\ldots, (y_{K-1}, x_K), (x_K, y_K)$ is bounded from below by some $\tau>0$ depending only on  $K$. Furthermore, $\mu_1$, $\mu_2$
are some probability measures.

Now,  let $m\in\mathbb N$ be so large that $(1-\alpha)^m\le\varepsilon (4||f||)^{-1}$. Since every element from $\{x_1,\ldots, x_K\}$ and $\{y_1,\ldots, y_K\}$ (on which the measures defining $\nu_1^1, \nu_2^1,\ldots,\nu_1^m, \nu_2^m$ are supported) is of the form $g_{\bf j} (z)$ for some ${\bf j}\in\Sigma_*$, we may find $\eta>0$ such that for any $w\in\mathbf S^1$ with $\d (z, w)<\eta$
we have
$$
\begin{aligned}
P^{N_1+\cdots +N_m}\delta_{w}&=\alpha P^{N_2+\cdots+N_m}\tilde\nu_2^1+\alpha (1-\alpha) P^{N_3+\cdots+N_m}\tilde\nu_2^2\\
&+\cdots+\alpha(1-\alpha)^{m-1}P^{N_m}\tilde\nu_2^m+(1-\alpha)^m\tilde\mu_2,
\end{aligned}
$$
and each pair of probability measures   $(\nu_1^j, \tilde\nu_2^j)$  is  a convex combination of pairs of measures that are uniformly distributed  over some collections of points $\{x_1,\ldots, x_K\}$ and $\{\tilde y_1,\ldots, \tilde y_K\}$, respectively, such that
$$
x_1<\tilde y_1<x_2<\tilde y_2<\cdots < x_K<\tilde y_K,
$$
and the length of the arcs $(x_1, \tilde y_1), (\tilde y_1, x_2),\ldots, (\tilde y_{K-1}, x_K), (x_K, \tilde y_K)$ is bounded from below by $\tau/2$.
From Remark~\ref{remark} it follows now that there exists $n_0\in\mathbb N$ such that for any $n\ge n_0$ and every $j\in \{1,\dots, m\}$ we have
$$
\left |\int_{\mathbf S^1} f(x)P^n\nu_1^j (\d x)-\int_{\mathbf S^1} f(x)P^n\nu_2^j (\d x)\right|\le \varepsilon/2.
$$
Hence, we obtain that for any $w\in\mathbf S^1$ such that $\d (w, z)<\eta$ we have
$|P^{* n} f(z)-P^{* n} f(w)|\le \varepsilon$ for $n\ge N_1+\ldots+N_m+n_0$ and we are done.
\end{proof}

\section{Strong Law of Large Numbers}

Let $\Gamma=\{g_1,\dots g_k\}\subset H^+$ and let  $\tilde\Gamma=\{g_1^{-1},\dots g_k^{-1}\}$. We define the  probability distribution on $\tilde\Gamma$ by putting  $p(g_i^{-1})=p_i$.

\begin{prop}\label{strong} Let $\Gamma=\{g_1,\dots g_k\}\subset H^+$ be admissible and let $p$ be a probability distribution. Then the  Strong Law of Large Numbers for trajectories starting from an arbitrary point  holds. More precisely: let $\phi\in C(\mathbf{S}^1)$. For every $x\in\mathbf{S}^1$ there exists a subset $\Omega'\subset \Omega$ with $\mathbb P(\Omega')=1$ such that for every $\omega=(i_1,i_2,\dots)\in \Omega'$ 
\begin{equation}\label{SLLN2_1}
\frac{\phi(g_{i_1}(x))+\phi(g_{i_2, i_1}(x))+\ldots+\phi(g_{i_n, i_{n-1},\cdots, i_1}(x))}{n}\to \int_{S^1}\int_{\Gamma}\phi(y)dp(g)\mu (dy).
\end{equation}

\end{prop}
\begin{proof} From Lemma \ref{Lem1.26.12.15} it follows that $\Gamma$ is either equicontinuous or contractive.
If $\Gamma$ is equicontinuous then the theorem  holds simply by the Birkhoff ergodic theorem. Indeed, since the invariant measure $\mu_*$ is unique it is also ergodic. From Birkhoff's theorem it follows then that formula \eqref{SLLN2_1} holds  for  $\mathbb P\times\mu_*$-almost every pair $(\omega, z)$.  Since the support of $\mu_*$ is equal to $\mathbf{S}^1$, for any $x\in\mathbf{S}^1$ and any $k\in\mathbb N$ we may find a point $z_k\in\mathbf{S}^1$ and a set $\Omega_k\subset\Omega$ with $\mathbb P(\Omega_k)=1$ such that condition \eqref{SLLN2_1} holds with $x$ replaced with $z_k$ and
$|\phi(g_{i_n, i_{n-1},\cdots, i_1}(x))-\phi(g_{i_n, i_{n-1},\cdots, i_1}(z_k))|<1/k$ for $n\in\mathbb N$ and $(i_1, i_2,\ldots)\in\Omega_k$, by the fact that $\Gamma$ is equicontinuous and $\varphi$ is uniformly continuous. Let $\Omega'=\bigcap_{k=1}^{\infty}\Omega_k$. Since $\mathbb P(\Omega')=1$ and $x$ satisfies condition \eqref{SLLN2_1} for $(i_1, i_2,\ldots)\in\Omega'$, we are done.

Now assume that $\Gamma$ is contractive and that the theorem does not hold. Then there exists $x\in\mathbf S^1$ and a set $\Omega_x\subset \Omega$ such that $\mathbb P(\Omega_x)>0$ and for every $\omega=(i_1,i_2,\dots)\in \Omega_x$ formula \eqref{SLLN2_1} does not hold. Taking a subset of $\Omega_x$, if necessary, we may assume that
\begin{equation}\label{eq1.17.03.15}
\limsup_{n\to\infty}\left|\frac{\phi(g_{i_1}(x))+\phi(g_{i_2, i_1}(x))+\ldots+\phi(g_{i_n, i_{n-1},\cdots, i_1}(x))}{n}-\int_{S^1}\int_{\Gamma}\phi(y)\d p(g)\mu_* (\d y)\right|>\varepsilon
\end{equation}
for all $\omega=(i_1,i_2,\dots)\in \Omega_x$ and some $\varepsilon>0$. 

Now consider the system generated by  $\tilde\Gamma$ and the probability distribution $p$.  Let $\tilde\mu_*$ be its unique invariant measure.
Since $\tilde\mu_*(\{x\})=0$ we may find $\delta>0$ and a subset $\tilde\Omega_x\subset\Omega_x$ with $\mathbb P(\tilde\Omega_x)>0$ such that $\supp(\omega(\tilde\mu_*))\cap(x-\delta, x+\delta)=\emptyset$ for 
$\omega\in\tilde\Omega_x$.

 Let $\theta>0$ be such that $|\phi(u)-\phi(v)|\le\varepsilon/2$ for $|u-v|<\theta$. Set $\gamma:=\inf_{x\in\mathbf S^1} \tilde\mu_*((x-\theta/2, x+\theta/2))$. Obviously $\gamma>0$.
Since $g_{i_1}^{-1}\circ g_{i_2}^{-1}\circ\cdots\circ g_{i_n}^{-1}\circ\tilde\mu_*$ converges weakly to $\omega(\tilde\mu_*))$ for $\mathbb P$-a.s. $\omega=(i_1,i_2,\dots)\in\tilde\Omega_x$, we have $\tilde\mu_*((g_{i_1}^{-1}\circ g_{i_2}^{-1}\circ\cdots\circ g_{i_n}^{-1})^{-1}(
(x-\delta, x+\delta))<\gamma$ for all $n$ sufficiently large. From this and from  the definition of $\gamma$ it follows that
$|(g_{i_1}^{-1}\circ g_{i_2}^{-1}\circ\cdots\circ g_{i_n}^{-1})^{-1}(
(x-\delta, x+\delta))|<\theta$ for $n$ sufficiently large and all $\omega=(i_1,i_2,\dots)\in\tilde\Omega_x$. This gives
$|g_{i_n, i_{n-1},\cdots, i_1}(u)-g_{i_n, i_{n-1},\cdots, i_1}(v)|\le\theta$ and consequently
$$
|\phi(g_{i_n, i_{n-1},\cdots, i_1}(u))-\phi(g_{i_n, i_{n-1},\cdots, i_1}(v))|\le\varepsilon/2
$$
for all $n$ sufficiently large and $u, v\in (x-\delta, x+\delta)$. Together with condition (\ref{eq1.17.03.15}), this gives that
$$
\limsup_{n\to\infty}\left|\frac{\phi(g_{i_1}(u))+\phi(g_{i_2, i_1}(u))+\ldots+\phi(g_{i_n, i_{n-1},\cdots, i_1}(u))}{n}-\int_{S^1}\int_{\Gamma}\phi(y)\d p(g)\mu_* (\d y)\right|>\varepsilon/2
$$
for $\omega=(i_1, i_2,\ldots)\in\tilde\Omega_x$ and all $u\in (x-\delta, x+\delta)$, contrary to the Birkhoff theorem. This contradiction completes the proof. 
\end{proof}

\begin{rem} In \cite{alseda_misiurewicz} the authors introduce the notion of an essentially contracting system,  prove this property for some special system of piecewise linear maps of an interval, and obtain several interesting consequences. The proof of Proposition~\ref{strong} shows, in particular, that our system is essentially contracting. Contrary to \cite{alseda_misiurewicz}, we do not need special estimates; we simply use the  properties of the "conjugate" system generated by $\tilde\Gamma$.  
\end{rem}

\begin{rem} Observe that the known criteria for the Central Limit Theorem and Law of the Iterated Logarithm require more than it was proved in Theorem \ref{MTh}. To apply the results by M. Maxwell and M. Woodroofe (see \cite{MW}) we need to know the rate of convergence to the invariant measure. On the other hand, finding new sufficient conditions for the Central Limit Theorem in the setting of considered IFS's would be an interesting question worthy of further study.
\end{rem}

\end{document}